\documentclass{article}
\usepackage{amsfonts,amssymb,amsmath,amsthm}
\usepackage{url}

\def\thtext#1{
  \catcode`@=11
  \gdef\@thmcountersep{. #1}
  \catcode`@=12
}

\def\threst{
  \catcode`@=11
  \gdef\@thmcountersep{.}
  \catcode`@=12
}

\theoremstyle{plain}
\newtheorem{thm}{Theorem}[section]
\newtheorem{prop}[thm]{Proposition}
\newtheorem{cor}[thm]{Corollary}

\theoremstyle{definition}
\newtheorem{dfn}[thm]{Definition}
\newtheorem{rk}[thm]{Remark}
\newtheorem{constr}[thm]{Construction}
\newtheorem{agree}[thm]{Agreement}

 \catcode`@=11
 \def\.{.\spacefactor\@m}
 \catcode`@=12

\def\R{\mathbb R}

\def\e{\varepsilon}
\def\dl{\delta}

\def\r{\rho}
\def\s{\sigma}

\def\0{\emptyset}
\def\:{\colon}
\def\rom#1{\emph{#1}}
\def\({\rom(}
\def\){\rom)}
\def\sm{\setminus}
\def\ss{\subset}

\def\x{\times}

\def\bR{{\bar R}}
\def\bsg{{\bar\sigma}}
\def\bx{{\bar x}}
\def\by{{\bar y}}

\def\dis{\operatorname{dis}}
\def\opt{{\operatorname{opt}}}

\def\cD{{\cal D}}
\def\cH{{\cal H}}
\def\cM{{\cal M}}
\def\cP{{\cal P}}
\def\cR{{\cal R}}

\begin{document}
 \title{Realizations of Gromov--Hausdorff Distance.}
\author{A.~Ivanov, S.~Iliadis, A.~Tuzhilin}
\date{}
\maketitle

\begin{abstract}
It is shown that for any two compact metric spaces there exists an ``optimal'' correspondence which the Gromov--Hausdorff distance is attained at. Each such correspondence generates isometric embeddings of these spaces into a compact metric space such that the Gromov--Hausdorff distance between the initial spaces is equal to the Hausdorff distance between their images. Also, the optimal correspondences could be used for constructing the shortest curves in the Gromov--Hausdorff space in exactly the same way as it was done by Alexander Ivanov, Nadezhda Nikolaeva, and Alexey Tuzhilin in~\cite{IvaNikolaevaTuz}, where it is proved that the Gromov--Hausdorff space is geodesic. Notice that all proofs in the present paper are elementary and use no more than the idea of compactness.
\end{abstract}

%%%%%%%%%%%%%%%%%%%%%%%%%%%%%%
\section*{Introduction}
\markright{\thesection.~Introduction}
%%%%%%%%%%%%%%%%%%%%%%%%%%%%%%
The Gromov--Hausdorff distance between metric spaces plays an important role in the Modern Mathematics, see~\cite{BurBurIva}. The restriction of this distance to the isometry classes of compact metric spaces generates the Gromov--Hausdorff metric space. It is well-known that the Gromov--Hausdorff space is Polish (i.e., it is complete and separable), path-connected and, as it was recently shown in~\cite{IvaNikolaevaTuz}, geodesic. To prove the latter result, the authors used the technic of correspondences together with the evident fact that for finite metric spaces there always exists a correspondence which the Gromov--Hausdorff distance is attained at. In the present paper we show that such ``optimal'' correspondences exist for arbitrary compact metric spaces. The latter, together with a standard procedure, enables to construct some metric space and isometric embeddings of the initial spaces into it, such that the Gromov--Hausdorff distance between the initial spaces equals to the Hausdorff distance between their images. Existence of such ``realization'' for a pair of compact metric spaces in terms of ultraproducts was mentioned in MathOverflow~\cite{blog}.

Repeating word-by-word the construction from~\cite{IvaNikolaevaTuz}, one can use each optimal correspondence to connect any two compact metric spaces by a shortest curve in Gromov--Hausdorff space.

When the present paper was completely prepared for publication, we found the paper~\cite{Memoli}, where some our results were independently obtained. Nevertheless, we decided to submit our paper ``as is,'' because it is slightly wider than~\cite{Memoli}, contains a construction of the realization, and also a few other technical results which might help in further investigation of the Gromov--Hausdorff space geometry.

%%%%%%%%%%%%%%%%%%%%%%%%%%%%%%
\section{Preliminaries}
\markright{\thesection.~Preliminaries}
%%%%%%%%%%%%%%%%%%%%%%%%%%%%%%
Throughout this paper $X$ and $Y$ are always nonempty sets.

Let $X$ be a metric space. The distance between its points $x,\,y\in X$ is denoted by $\r(x,y)=|xy|$.

Let $\cP(X)$ be the set of all \textbf{nonempty} subsets of the space $X$. For each $A,\,B\in\cP(X)$ and $x\in X$ we put
\begin{flalign*}
\indent&|xA|=|Ax|=\inf\bigl\{|xa|:a\in A\bigr\},&\\
\indent&d(A,B)=\sup\bigl\{|aB|:a\in A\bigr\}=\sup_{a\in A}\inf_{b\in B}|ab|,&\\
\indent&d_H(A,B)=\max\{d(A,B),d(B,A)\}=\max\bigl\{\sup_{a\in A}\inf_{b\in B}|ab|,\,\sup_{b\in B}\inf_{a\in A}|ba|\bigr\}.
\end{flalign*}

\begin{dfn}
The function $d_H$ is called the \emph{Hausdorff distance}.
\end{dfn}

\begin{rk}
In general, the Hausdorff distance is not a metric because it may vanish for unequal subsets. For instance, this is so for the segment $A=[a,b]$ and half-interval $B=[a,b)$, both are considered as subsets of the real line $\R$ with the standard distance function. Also, the Hausdorff distance may be infinite, say, between a point and a line.
\end{rk}

By $\cH(X)$ we denote the family of all nonempty closed bounded subsets of a metric space $X$.

\begin{thm}[\cite{BurBurIva}]
The function $d_H$ is a metric on $\cH(X)$.
\end{thm}

The following important property of the space $\cH(X)$ endowed with the Hausdorff metric will be useful.

\begin{thm}[\cite{BurBurIva}]\label{thm:Hausd_dist_comp}
The metric spaces $\cH(X)$ and $X$ are compact or non-compact simultaneously.
\end{thm}

Let $X$ and $Y$ be metric spaces. The triple $(X',Y',Z)$ consisting of a metric space $Z$ and its subsets $X'$ and $Y'$ isometric to $X$ and $Y$, respectively, is called a \emph{realization of the pair $(X,Y)$}. \emph{The Gromov--Hausdorff distance $d_{GH}(X,Y)$ between $X$ and $Y$} is the infimum of those $r$ for which there exists a realization $(X',Y',Z)$ of the pair $(X,Y)$ such that $d_H(X',Y')\le r$.

It turns out that under the Gromov--Hausdorff distance between spaces $X$ and $Y$ calculation, one can consider a rather small set of realizations of the pair $(X,Y)$.  By $\cD(X,Y)$ we denote the set of all pseudometrics $\r$ on $X\sqcup Y$ such that their restrictions onto $X$ and $Y$ coincide with the initial metrics of these spaces. Let $(X\sqcup Y)/\r$ denote the metric space obtained from $X\sqcup Y$ by identifying the points on zero $\r$-distance. Notice that the restrictions of the canonical projection $\pi\:X\sqcup Y\to(X\sqcup Y)/\r$ onto $X$ and $Y$ are isometric embeddings. We put $d_H(X,Y,\r)=d_H\bigl(\pi(X),\pi(Y)\bigr)$.

\begin{thm}[\cite{BurBurIva}]\label{thm:dGHasInfHasudorffForPseudometrics}
Under the above notations, we have
$$
d_{GH}(X,Y)=\inf_{\r\in\cD(X,Y)}d_H(X,Y,\r).
$$
\end{thm}

\begin{thm}[\cite{BurBurIva}]
The Gromov--Hausdorff distance is a metric on the isometry classes of compact metric spaces.
\end{thm}

\begin{dfn}
The set of isometry classes of compact metric spaces endowed with the Gromov--Hausdorff metric is called the \emph{Gromov--Hausdorff space\/} and is denoted by $\cM$.
\end{dfn}

Now we describe a convenient way to calculate Gromov--Hausdorff distance. Recall that \emph{a relation\/} between sets $X$ and $Y$ is a subset of the Cartesian product $X\x Y$.  By $\cP(X,Y)$ we denote the set of all \textbf{nonempty\/} relations between $X$ and $Y$.

\begin{dfn}
A relation $R\ss X\x Y$ between $X$ and $Y$ is called \emph{a correspondence}, if the restrictions onto $R$ of the canonical projections $\pi_X\:(x,y)\mapsto x$ and $\pi_Y\:(x,y)\mapsto y$ are surjective. The set of all correspondences between $X$ and $Y$ is denoted by $\cR(X,Y)$.
\end{dfn}

Notice that any $\s\in\cP(X,Y)$ such that there exists an $R\in\cR(X,Y)$, $R\ss\s$, satisfies $\s\in\cR(X,Y)$; in particular, if $X$ and $Y$ are topological spaces, then the closure $\bR$ of a correspondence $R\in\cR(X,Y)$ is a correspondence inself.

Now, let $\s$ be a nonempty relation between pseudometric spaces $X$ and $Y$.

\begin{dfn}
\emph{The distortion $\dis\s$ of the relation $\s$} is the value
$$
\dis\s=\sup\Bigl\{\bigl||xx'|-|yy'|\bigr|: (x,y)\in\s,\ (x',y')\in\s\Bigr\}.
$$
\end{dfn}

Obviously, if $\0\ne\s_1\ss\s_2\in\cP(X,Y)$, then $\dis\s_1\le\dis\s_2$.

\begin{prop}\label{prop:dis_closure}
If $\bsg$ is the closure of $\s\in\cP(X,Y)$, then $\dis\bsg=\dis\s$.
\end{prop}

\begin{proof}
Since $\dis\s\le\dis\bsg$, it suffices to verify the converse inequality. Suppose the contrary: $\dis\s<\dis\bsg$. Notice that this assumption implies $\dis\s<\infty$.

By definition, for each $\e>0$ and any $(\bx,\by),\,(\bx',\by')\in\bsg$ there exist $(x,y),\,(x',y')\in\s$ such that $\bigl||\bx\bx'|-|xx'|\bigr|<\e/3$, $\bigl||\by\by'|-|yy'|\bigr|<\e/3$, and thus
$$
\bigl||\bx\bx'|-|\by\by'|\bigr|<\bigl||xx'|-|yy'|\bigr|+2\e/3\le\dis\s+2\e/3.
$$
Passing to the supremum, we conclude that $\dis\bsg<\dis\s+2\e/3$. Since $\e$ is arbitrary, we have $\dis\bsg\le\dis\s$, a contradiction.
\end{proof}

By $\cP_c(X,Y)$ we denote the set of all closed nonempty relations between $X$ and $Y$; similarly, let $\cR_c(X,Y)$ stand for the set of all closed correspondences between $X$ and $Y$.

\begin{thm}[\cite{BurBurIva}]\label{th:GH-metri-and-relations}
For any metric spaces $X$ and $Y$ we have
$$
d_{GH}(X,Y)=\frac12\inf\bigl\{\dis R:R\in\cR(X,Y)\bigr\}.
$$
\end{thm}

The next result follows immediately from~\ref{prop:dis_closure} and~\ref{th:GH-metri-and-relations}.

\begin{cor}\label{cor:dis_closed}
For and $X$ and $Y$ we have
$$
d_{GH}(X,Y)=\frac12\inf\bigl\{\dis R:R\in\cR_c(X,Y)\bigr\}.
$$
\end{cor}

The next construction establishes a link between correspondences from $\cR(X,Y)$ and pseudometrics on $X\sqcup Y$.

\begin{constr}
Consider arbitrary metric spaces $X$ and $Y$, and some correspondence $R\in\cR(X,Y)$. Suppose that $\dis R<\infty$. Extend the metrics of $X$ and $Y$ upto a symmetric function $\r_R$ defined on pairs of points from $X\sqcup Y$: for $x\in X$ and $y\in Y$ put
$$
\r_R(x,y)=\r_R(y,x)=\inf\bigl\{|xx'|+|yy'|+\frac12\dis R:(x',y')\in R\bigr\}.
$$
\end{constr}

\begin{thm}\label{thm:correspondence-to-pseudometric}
The function $\r_R$ is a pseudometric, and $d_H(X,Y,\r_R)=\frac12\dis R$.
\end{thm}

%%%%%%%%%%%%%%%%%%%%%%%%%%%%%%
\section{Realization of Gromov--Hausdorff distance}
\markright{\thesection.~Realization of Gromov--Hausdorff distance}
%%%%%%%%%%%%%%%%%%%%%%%%%%%%%%

Let $X$ and $Y$ be arbitrary metric spaces.

\begin{agree}
In what follows, when we work with the metric space $X\x Y$, we always suppose that its distance function is
$$
\bigl|(x,y)(x',y')\bigr|=\max\bigl\{|xx'|,\,|yy'|\bigr\},
$$
and just this metric generates the Hausdorff distance on $\cP(X\x Y)$. Thus, the space $\cP(X,Y)=\cP(X\x Y)$ and all its subspaces (in particular, $\cR(X,Y)$) are supposed to be endowed with the distances of $\cP(X\x Y)$.
\end{agree}

\begin{rk}
If $X,\,Y\in\cM$, then $X\x Y\in\cM$, $\cP_c(X,Y)=\cH(X\x Y)$, and, by~\ref{thm:Hausd_dist_comp}, we have $\cP_c(X,Y)\in\cM$.
\end{rk}

\begin{prop}\label{prop:Rc-compact}
For $X,\,Y\in\cM$ the set $\cR_c(X,Y)$ is closed in $\cP_c(X,Y)$, thus, $\cR_c(X,Y)\in\cM$.
\end{prop}

\begin{proof}
It is suffice to show that for each $\s\in\cP_c(X,Y)\sm\cR_c(X,Y)$ there exists a neighborhood $U$ which does not intersect $\cR_c(X,Y)$. Since $\s\not\in\cR(X,Y)$, then either $\pi_X(\s)\ne X$, or $\pi_Y(\s)\ne Y$, where $\pi_X$ and $\pi_Y$ are the canonical projections. To be definite, suppose that the first condition holds, i.e., there exists $x\in X\sm\pi_X(\s)$. Since $\s$ is a closed subset of the compact $X\x Y$, then it is compact itself, and therefore, $\pi_X(\s)$ is compact in $X$, thus, $\pi_X(\s)$ is closed. The latter implies that there exists an open ball $U_\e(x)$ such that $U_\e(x)\cap\pi_X(\s)=\0$. So, one can take $U_\e(x)\x Y$ for $U$.
\end{proof}

Define a function $f\:(X\x Y)\x(X\x Y)\to\R$ as $f(x,y,x',y')=\bigl||xx'|-|yy'|\bigr|$. Clearly that $f$ is continuous. Notice that for each $\s\in\cP(X,Y)$ we have
$$
\dis\s=\sup\bigl\{f(x,y,x',y'):(x,y),\,(x',y')\in\s\bigr\}=\sup f|_{\s\x\s}.
$$

\begin{prop}\label{prop:dis_continuous}
If $X,\,Y\in\cM$, then the function $\dis\:\cP_c(X,Y)\to\R$ is continuous.
\end{prop}

\begin{proof}
Since $(X\x Y)\x(X\x Y)$ is compact, then the function $f$ is uniformly continuous, thus for any $\s\in\cP_c(X,Y)$ and any $\e>0$ there exists $\dl>0$ such that for the open ball $U=U_\dl^{X\x Y}(\s)\ss X\x Y$ of radius $\dl$ centered at $\s$ it holds
$$
\sup f|_{U\x U}\le\sup f|_{\s\x\s}+\e.
$$
By $V$ we denote the open ball $U_\dl^{\cP_c(X,Y)}(\s)\ss\cP_c(X,Y)$ of radius $\dl$ centered at $\s$. Since for any $\s'\in V$ we have $\s'\ss U$, then it follows that
$$
\dis\s'=\sup f|_{\s'\x\s'}\le\sup f|_{U\x U}\le\sup f|_{\s\x\s}+\e=\dis\s+\e.
$$
Swapping $\s$ and $\s'$, we get $|\dis\s-\dis\s'|\le\e$, and hence, the function $\dis$ is continuous.
\end{proof}

\begin{dfn}
A correspondence $R\in\cR(X,Y)$ is called \emph{optimal\/} if $d_{GH}(X,Y)=\frac12\dis R$. By $\cR_\opt(X,Y)$ we denote the set of all optimal correspondences between $X$ and $Y$.
\end{dfn}

\begin{thm}\label{thm:optimal-correspondence-exists}
For any $X,\,Y\in\cM$ we have $\cR_\opt(X,Y)\ne\0$.
\end{thm}

\begin{proof}
By~\ref{prop:dis_continuous}, the function $\dis\:\cR_c(X,Y)\to\R$ is continuous, and by~\ref{prop:Rc-compact} the space $\cR_c(X,Y)$ is compact, thus $\dis$ attains its least value, half of which equals $d_{GH}(X,Y)$ by~\ref{cor:dis_closed}. Therefore, each correspondence $R$ which this least value is attained at is optimal.
\end{proof}

From~\ref{thm:optimal-correspondence-exists} and \ref{thm:correspondence-to-pseudometric} we immediately get the following result.

\begin{cor}\label{cor:realization}
For each $X,\,Y\in\cM$ there exists
\begin{enumerate}
\item a correspondence $R\in\cR(X,Y)$ such that $d_{GH}(X,Y)=\frac12\dis R$\rom;
\item a pseudometric $\r$ on $X\sqcup Y$ such that $d_{GH}(X,Y)=d_H(X,Y,\r)$\rom;
\item a metric space $Z$ and isometric embeddings of $X$ and $Y$ into $Z$ such that the Hausdorff distance between their images equals $d_{GH}(X,Y)$.
\end{enumerate}
\end{cor}

%%%%%%%%%%%%%%%%%%%%%%%%%%%%%%
\section{Optimal correspondences and shortest curves}
\markright{\thesection.~Optimal correspondences and shortest curves}
%%%%%%%%%%%%%%%%%%%%%%%%%%%%%%

In~\cite{IvaNikolaevaTuz}, for each pair of finite metric spaces, using an optimal correspondence between these spaces which evidently exists, we constructed a shortest curve in $\cM$ connecting those spaces. To prove the existence of shortest curves for arbitrary compact metric spaces, we used the limiting process and Gromov precompact criterion. Now, \ref{cor:realization} enables us to construct the shortest curve directly by means of the optimal correspondences. Repeating word-by-word the proof of the main result from~\cite{IvaNikolaevaTuz}, we get the following theorem.

\begin{thm}
For any $X,\,Y\in\cM$ and each $R\in\cR_\opt(X,Y)$ the family $R_t$, $t\in[0,1]$, of compact metric spaces such that $R_0=X$, $R_1=Y$, and for $t\in(0,1)$ the space $R_t$ is equal to $(R,\r_t)$, where
$$
\r_t\bigl((x,y),(x',y')\bigr)=(1-t)|xx'|+t\,|yy'|,
$$
is a shortest curve in $\cM$ connecting $X$ and $Y$.
\end{thm}

\end{document}